\documentclass[A4paper,12pt]{article}
\usepackage[utf8]{inputenc}
\usepackage{tikz}

\usepackage{amsmath,amsthm,amscd,amssymb,eucal,mathrsfs}
\usepackage{amsfonts}
\usepackage{latexsym}
\usepackage{graphicx} 
\usepackage{fullpage}
\usepackage{multicol}
\usepackage{dsfont}
\usepackage{natbib}
\usepackage[all]{xy}
\usepackage{multirow}
\usepackage{color}

\newtheorem{thm}{Theorem}[section]

\newtheorem{cor}[thm]{Corollary}
\newtheorem{remark}[thm]{Remark}
\newtheorem{definition}[thm]{Definition}

\newtheorem{example}[thm]{Example}

\newtheorem{Proposition}[thm]{Proposition}

\newtheorem*{Satz*}{Satz}

\newtheorem{Lemma}[thm]{Lemma}

\newcommand{\mathset}[1]{{\left\{#1\right\}}}
\newcommand{\absolute}[1]{\left\lvert#1\right\rvert}
\newcommand{\norm}[1]{\left\|#1\right\|}

\DeclareMathOperator{\closure}{cl}

\DeclareMathOperator{\dom}{dom}

\DeclareMathOperator{\supp}{supp}

\title{Neumann Boundary Value Problems on $p$-Adic Analytic Manifolds}
\author{Patrick Erik Bradley}
\date{\today}

\begin{document}

\maketitle

\begin{abstract}
Inspired by graph theory, Neumann Boundary Value Problems on $p$-adic analytic manifold are formulated. Existence and uniqueness of weak solutions is proven in the ultracontractive case. For this task, a $p$-adic Gauss-Green identity and an extended eigenvalue formula for wavelets with small supports on the manifold are proven. As an application, the closeness to a zero of an analytic function on the manifold is shown to be detectable via a Neumann Boundary Value Problem.
\end{abstract}

\section{Introduction}

Boundary Value Problems have been studied for a long time, leading also to many different applications, cf.\ e.g.\ \cite{Taira2009} for applications in the theory of Markov processes. The ones we have in mind here apartain to the study of geometric properties of $p$-adic analytic manifolds through Laplacian integral operators, and their corresponding Neumann Boundary Value Problems. Over $p$-adic domains, boundary value problems  are only recently coming into the focus of research, highlighted with Kochubei's study of $p$-adic Dirichlet problems in the context of H\"older regularity of boundaries \cite{Kochubei2023}. Other such research deals with boundary value problems in mathematical physics formulated over the Bruhat-Tits tree \cite{GJT2019,HJ2026}. Elliptic operators over compact domains in $p$-adic $n$-space lead the author to formulate boundary value problems as extensions of corresponding ones in graph theory \cite{EllipticBVP_p}. 
The construction of Laplacian integral operators on compact $p$-adic analytic manifolds in \cite{DiffMfp,HearingSerre} opened a path towards novel boundary value problems on such domains \cite{bvp_mfp}. Here, it is Neumann-type boundary conditions which are now being focused on for the first time, at least to the author's awareness.  
\newline

The goal of this article is to locate zeros of a regular analytic function $f\colon X\to K$ on a $p$-adic analytic $n$-manifold $X$ defined over a non-archimedean local field $K$ via solving Neumann Boundary Value Problems. The idea is to state the boundary condition as having the $p$-adic analogue of a normal derivative introduced here, of a solution to the Laplace equation for a certain integral operator equal to a wavelet supported on a ball centred in a point $a\in X$ near a zero of $f$. In particular the rate of decay of the solution will be studied as $a$ moves closer to the zero. It turns out that this can be detected, as well as the order of the zero of $f$ which the point $a$ is approaching, and also if $a$ is not moving closer to a zero. This is the final theorem of this article, called \emph{Closeness to zeros}. 
\newline

The main theoretical result is that a $p$-adic Laplacian integral operator whose associated semigroup is ultracontractive, then the Neumann Boundary Value Problem (NBVP) for a pair $(\phi,\Omega)$ consisting of an $L^\infty$-function $\phi$ on the boundary of an open subset $\Omega$ of $X$ (the boundary condition), such that $\phi$ has zero mean, has a unique weak solution, which can be obtained via integrating   $\phi$ times the Green function associated with the operator,  over the boundary of $\Omega$. This is Theorem 5.6.  
\newline

The first link between the two results is provided by a $p$-adic Gauss-Green identity which links  Dirichlet forms defined using an open $\Omega\subseteq X$ and its graphon-theoretic closure with its associated operator and normal derivative, cf.\ Lemma 3.3. 
\newline

The second link is given
by an eigenvalue formula for wavelets with small support,  extending the one in \cite{DiffMfp} and \cite{HearingSerre} by including the absolute value of $f$ in the measure. This is a generalisation in that now the differential $n$-form giving the Radon measure on $X$ is ot everywhere non-vanishing, cf.\ Proposition 6.1.
\newline

Some notation is in order: $K$ is a non-archimedean local field with ring of integers $O_K$, and whose Haar measure is $\mu$, appearing as  $\absolute{dx}$ or $\absolute{dy}$  in integrals (but not only). The uniformiser of $K$ is denoted as $\pi\in O_K$.
\newline

The following Section 2 recalls some notions from $p$-adic analytic manifolds. This is followed by Section 3 which proves the $p$-adic Gauss Green identity. Section 4 formulates and proves the NBVP of this article. Section 5 studies the heat kernel and expresses the Green function in terms of eigenfunctions in the ultracontractive case. The last Section 6 is devoted to the final result about detecting closeness to zeros of an analytic  function on $X$.

\section{Preliminaries on $p$-adic analytic manifolds}

Analytic manifolds can be  defined over a  non-archimedean local field $K$ in an analogous manner as over the reals, taking into account that the transition functions between overlapping charts are $K$-analytic. This is laid out in e.g.\ \cite{Igusa2001,Serre1992,Schneider2011,WeilAAG}. The book \cite{Taira1998} about manifolds over the reals is also helpful.
\newline

In order to be able to construct integral Laplacian operators over their real- and complex-valued function spaces, it was found useful to use a connected nerve complex coming from a suitable atlas, which in the compact case can be assumed finite \cite{DiffMfp}.
In order to be able to define an analogon of geodetic distance (called $p$-adic geodetic distance), a Radon measure given by a nowhere vanishing analytic differential $n$-form is useful, where $n$ is the (constant) dimension of the $p$-adic analytic manifold $X$.
If $X$ is compact, then it is known that such an $n$-form does exist, cf.\ \cite[Th\'eor\`eme (2)]{Serre1965}.
\newline

The reason why we are not giving an explicit recall of the definition of $p$-adic geodetic distance, is that it is not explicitly used in this article. The only thing to keep in mind about it, is that it is locally $p$-adic distance given on charts by the maximum norm on $K^n$. Not every $p$-adic analytic manifold allows this, but if it has an integral structure (as introduced e.g.\ in \cite{BKL2026}), then the transition maps between overlapping charts take balls to balls of equal radius. And that is very helpful for defining the
$p$-adic geodetic distance. In \cite{DiffMfp}, this was done in the compact case with the so-called \emph{equalising} property of transition maps. Integral structures for this task are used in \cite{brad_Tamagawa_p}.
\newline

The philosophy in much of the author's work is to generalise notions from graph theory to the theory of $p$-adic analytic manifolds.
For example, weighted edges on a graph are given by discrete kernel functions, and thus here, we take more general kernel functions with input data $X\times X$, where $X$ is a $p$-adic analytic manifold. We call this structure a \emph{$p$-adic graphon}. And a $p$-adic graphon is connected, if any two points $x,y\in X$ can be connected by a sequence $x=x_0,x_1,\dots,x_n=y$ with
\[
w(x_i,x_{i+1})\neq0
\]
for $i=0,\dots,n-1$, where $w(x,y)$ is the kernel function at hand. Here, the article \cite{HS2022} turned out quite helpful for this task.
\newline

A note on the
varying sloppiness in the notation of $n$-forms on charts, in particular when integrating, is in order: the reference \cite{HearingSerre} explains various kinds of more or less sloppy notation, in the sense that the chart maps are fully, partially, or not at all included in the notation. An example attempting to be readable for beginners is \cite{DiffMfp}, whereas the usual fully sloppy notation is preferred in \cite{brad_Tamagawa_p}, as is quite often found among  number theorists. Since we have learned in \cite{HearingSerre} how to read that kind of notation, we will employ it here, unless we see that it is indeed helpful for comprehension of a proof.

\section{A $p$-adic Gauss-Green Identity}

Let $X$ be a compact $p$-adic analytic manifold with integral structure $\Lambda$, $\mathcal{A}$ a finite 
$O_K$-compatible atlas on $X$ giving rise to $\Lambda$, and such that the nerve complex 
$N(A)$ is connected. 
Let $w(x,y) \ge 0$ be a symmetric and non-negative kernel function on $X$ which can be thought of
in many cases as
\[
w(x, y) = g(d_\Lambda(x,y))
\]
for $x,y\in X$, $d_\Lambda(x,y)$ the $p$-adic geodetic distance, and in any case is locally constant on $X^2$ outside the diagonal. Also, in the general case, it is assumed that $w(x,y)$ defines a connected graphon structure on $X$.
\newline

In the following, Dirichlet forms on spaces of complex-valued  functions on the $p$-adic analytic manifold $X$ are used. An exposition of a general theory of Dirichlet forms can be found in \cite{Fukushima1980}.
\newline

The following Dirichlet form is of interest:
\begin{align}
\mathcal{E}_\Lambda(u,v)
&=
\frac12\int_X\int_X w(x, y)
\left(\overline{u(x)}- \overline{u(y)}\right)
(v(x) -v(y))\,d\mu_\Lambda(x)\,d\mu_\Lambda(y)
\\
&\stackrel{(*)}{=} \langle u, \Delta_w v\rangle_{L^2}
\end{align}
for $u,v\in L^2(X,\mu_\Lambda)$, and with Laplacian operator
\[
\Delta_w f (x) =
\int_X w(x, y)(f(x) - f (y))\,d\mu_\Lambda (y)
\]
with $f\in\mathcal{D}(X)$. The subscript $\Lambda$ indicates a dependence on the choice of
an integral structure $\Lambda$ on $X$. 
\newline

Notice that $(*)$ requires a proof.

\begin{Lemma}\label{DirichletLaplace}
It holds true that
$
\mathcal{E}_\Lambda(u,v)=\langle u,\Delta_wv\rangle_{L^2}
$
for $u,v\in L^2(X,\mu_\Lambda)$.
\end{Lemma}

\begin{proof}
The proof extends from the discrete setting
in \cite[Proposition 1.12]{JP2013} as follows:
Let $u,v\in\dom(\mathcal{E}_\Lambda)$. Then
\begin{align*}
\langle u,\Delta_w v\rangle_{L^2}&=\frac12\int_X\overline{u(x)}\Delta_wv(x)\,d\mu_\Lambda(x)+\frac12\int_X\overline{u(y)}\Delta_wv(y)\,d\mu_\Lambda(y)
\\
&=\frac12\int_X\int_X\overline{u(x)}w(x,y)(v(x)-v(y))\,d\mu_\Lambda(y)\,d\mu_\Lambda(x)
\\
&-\frac12\int_X\int_X\overline{u(y)}w(x,y)(v(x)-v(y))\,d\mu_\Lambda(x)\,d\mu_\Lambda(y)
\\
&=\frac12\int_X\int_Xw(x,y)\left(\overline{u(x)}-\overline{u(y)}\right)
(v(x)-v(y))\,d\mu_\Lambda(y)\,d\mu_\Lambda(x)
\\
&=\mathcal{E}_\Lambda(u,v)\,,
\end{align*}
as asserted.
\end{proof}

The following notation is standard:
\[
\mathcal{E}_\Lambda(u):=\mathcal{E}_\Lambda(u,u)
\]
for $u\in\dom(\mathcal{E}_\Lambda)$.
\newline

A \emph{contraction} of a function $u$ is a function $[u]$ such that
\[
\absolute{[u](x) - [u](y)}\le\absolute{u(x)-u(y)}\,,
\]
whenever $w(x, y)\neq 0$ 
with $x,y \in X$. A function 
$u \in L^2(X,\mu_\Lambda)$ is of \emph{finite
energy}, if 
$u\in\dom(\mathcal{E}_\Lambda)$, and 
$u$ is \emph{harmonic}, if 
$\Delta_w = 0$. 
A Dirichlet form $\mathcal{E}$ on
a Hilbert space $(H, \langle\cdot,\cdot\rangle_H)$ is \emph{closed}, if the norm
\[
\langle\cdot,\cdot\rangle_{\mathcal{E}}:=
\mathcal{E}(\cdot,\cdot)+\langle \cdot,\cdot\rangle_H
\]
defines a Hilbert space structure on $\dom(\mathcal{E})$.

\begin{Lemma}
The following statements hold true:
\begin{enumerate}
\item $\ker\mathcal{E}_\Lambda$ consists precisely of the constant functions $X\to\mathcal{C}$.
\item $\mathcal{E}_\Lambda(u, v) =\frac14 [\mathcal{E}_\Lambda(u + v) - \mathcal{E}_\Lambda(u - v)]$.
\item  If $[u]\in L^2(X,\mu_\Lambda)$ is a contraction of $u\in\dom(\mathcal{E}_\Lambda)$, then $\mathcal{E}_\Lambda[u]\le \mathcal{E}_\Lambda(u)$.
\item A harmonic function $h\in L^2 (X,\mu_\Lambda)$ of finite energy is constant.
\item The Dirichlet form $\mathcal{E}_\Lambda$ is closed.
\end{enumerate}
\end{Lemma}

Notice that $\mathcal{E}_\Lambda[u]$ is used here as a short-hand for $\mathcal{E}_\Lambda([u])$. Property 2.\ is
called \emph{polarisation}, and Property 3.\ \emph{Markov Property}.

\begin{proof}
1. This follows from the fact that the graphon defined by the kernel
function $w(x, y)$ is connected.

\smallskip\noindent
2. This is a straightforward calculation.

\smallskip\noindent
3. This is immediate.

\smallskip\noindent
4. This follows from Lemma \ref{DirichletLaplace} together with harmonicity as $\Delta_w h = 0$ and by
connectivity of the graphon $(X, w)$.

\smallskip\noindent
5. The proof first uses Lemma \ref{DirichletLaplace}, and then one observes that the following:
the map 
\[
\dom(\mathcal{E}_\Lambda)\to\mathds{R}_+\,,\;
u\mapsto\norm{u}_{\mathcal{E}_\Lambda}=
\langle u,u\rangle_{\mathcal{E}_\Lambda}
\]
is a norm. The completedness is shown as in \cite[Theorem 5.2.2]{Evans2010}. Namely, let $u_n\in\dom(\mathcal{E}_\Lambda)$ be a Cauchy sequence. Then $u_n,\Delta_wu_n$ are also a Cauchy sequence in $L^2(X,\mu_\Lambda)$. Thus, let $f_0,f_1\in L^2(X,\mu_\Lambda)$ be the limits of $u_n$, $\Delta_wu_n$, respectively. In order to see that $u\in\dom(\mathcal{E}_\Lambda)$, let $\phi\in\mathcal{D}(X)$. Then
\begin{align*}
\int_X\Delta_w^\ell u\phi\,d\mu_\Lambda&=\int_X u\Delta_w^\ell\phi\,d\mu_\Lambda
=\lim\limits_{n\to\infty}u_n\Delta_w^\ell\phi\,d\mu_\Lambda
=
\int_Xf_\ell\phi\,d\mu_\Lambda
\end{align*}
for $\ell=0,1$. This proves the completeness of $(\dom(\mathcal{E}_\Lambda),\norm{\cdot}_{\mathcal{E}_\Lambda})$. Hence, the Dirichlet form $\mathcal{E}_\Lambda$ is closed. All assertions are now proven.
\end{proof}

The main part of this article will be concerned with boundary value problems on an open subset $\Omega\subset X$. Using the kernel function $w(x,y)$ on $X$, allows to define the following subsets of $X$:
\begin{align*}
\delta_w\Omega&=\mathset{y\in X\setminus \Omega\mid\exists x\in\Omega\colon w(x,y)\neq 0}
\\
\closure_w\Omega&=\Omega\cup\delta_w\Omega\,.
\end{align*}
These are similar constructions as in e.g.\ \cite{EllipticBVP_p}.
\newline

Following ideas in \cite{Kasue2017}, define for an open $\Omega\subseteq X$ the following Dirichlet
form:
\[
\mathcal{E}_{\Omega,\closure_w\Omega}(u,v)=\frac12\int_\Omega\int_{\closure_w\Omega}w(x,y)\left(\overline{u(x)}-\overline{u(y)}\right)(v(x)-v(y))\,d\mu_\Lambda(y)\,d\mu_\Lambda(x)
\]
for $u,v\in L^2(\closure_w\Omega,\mu_\Lambda)$. The normal derivative w.r.t.\ $\delta_w\Omega$ is defined as
follows:
\[
N_{\delta_w\Omega}v(\xi)=\int_\Omega w(\xi,z)(v(\xi)-v(z))\,d\mu_\Lambda(z)
\]
for $\xi\in\delta_w\Omega$. 
Define the Dirichlet form
\[
\mathcal{E}_\Omega(u,v)=\frac12\int_\Omega\int_\Omega
w_\Omega(x,y)\left(\overline{u(x)}-\overline{u(y)}\right)(v(x)-v(y))\,d\mu_\Lambda(y)\,d\mu_\Lambda(x)
\]
and the
operator
\[
\Delta_\Omega u(x)=\int_X w_\Omega(x,y)(u(x)-u(v))\,d\mu_\Lambda(y)
\]
with kernel function
\[
w_\Omega\colon X\times X\to\mathds{R}_+\,,\;(x,y)\mapsto\begin{cases}
w(x,y),&x,y\in\closure_w\Omega
\\
0,&\text{otherwise}
\end{cases}
\]
for $u,v\in L^2(\Omega,\mu_\Lambda)$, where the same notation is used for the restricted measure.

\begin{Lemma}[$p$-adic Gauss-Green identity]\label{GaussGreen_p}
It holds true that
\begin{align*}
\mathcal{E}_{\Omega,\closure_w\Omega}(u,v)&=\int_\Omega\overline{u(x)}\Delta_\Omega v(x)\,d\mu_\Lambda(x)
+\int_{\delta_w\Omega}\overline{u(\xi)}N_{\delta_w\Omega}v(\xi)\,d\mu_\Lambda(\xi)
\\
&=\langle u,\Delta_\Omega v\rangle_{L^2(\Omega,\mu_\Lambda)}
+\langle u,N_{\delta_w\Omega}v\rangle_{L^2(\delta_w\Omega,\mu_\Lambda)}
\end{align*}
for $u,v\in\dom(\mathcal{E}_{\Omega,\closure_w\Omega})\subset L^2(\closure_w\Omega,\mu_\Lambda)$.
\end{Lemma}

\begin{proof}
By writing
\[
\mathcal{E}_{\Omega,\closure_w\Omega}(u,v)=\mathcal{E}_\Omega(u,v)+\frac12\int_\Omega\int_{\delta_w\Omega} w(x,y)\left(\overline{u(x)}-\overline{u(y)}\right)(v(x)-v(y))\,d\mu_\Lambda(y)\,d\mu_\Lambda(x)
\]
and then by adapting the proof of Lemma \ref{DirichletLaplace} to both summands, the asserted identity follows.
\end{proof}

\section{$p$-adic Neumann boundary value problems}

Generalising \cite{HS2022}, we can define a  of Neumann Boundary Value Problem in the context of $p$-adic analytic manifolds:

\begin{definition}
Let $\phi\in L^2(\delta_w\Omega,\mu_\Lambda)$. A function 
$u\in\dom(\mathcal{E}_{\closure_w\Omega})$ solves
the \emph{Neumann Boundary Value Problem} for $(\phi,\Omega)$, if
\begin{align}\label{NBVP}
\Delta_\Omega u(x) = 0\quad\text{and}\quad N_{\delta_w\Omega} u(y) = \phi(y)
\end{align}
for $x\in\Omega$ and $y\in\delta_w\Omega$. It is a \emph{weak solution} to the Neumann Boundary Value
Problem for $(\phi,\Omega)$, if
\begin{align}\label{NBVP_w}
\langle\Delta_\Omega u, v\rangle_{L^2} = 0\quad\text{and}\quad
\langle N_{\delta_w\Omega}u,v\rangle_{L^2} = \langle\phi,v\rangle_{L^2}
\end{align}
for all $v\in\dom(\mathcal{E}_{\closure_w\Omega})$ having support in $\Omega$ and $\delta_w\Omega$, respectively.
\end{definition}

\begin{Lemma} Two solutions (weak or not) of the Neumann Boundary Value Problem (\ref{NBVP})
for $(\phi,\Omega)$ differ by an element of $\ker(\mathcal{E}_\Omega)$. In particular, if $\Omega$ has the structure
of a connected graphon induced by 
$w_\Omega$, then they differ only by a constant.
\end{Lemma}

\begin{proof}
If $u,v \in\dom(\mathcal{E}_{\closure_w\Omega})$ are two solutions of (\ref{NBVP}) of (\ref{NBVP_w}) for $(\phi,\Omega)$, then $g = u - v$ is a solution of (\ref{NBVP}) or (\ref{NBVP_w}) for $(0,\Omega)$. Hence, by the $p$-Adic Gauss-Green
Identity (Lemma \ref{GaussGreen_p}), we have in both cases that
\[
\mathcal{E}_{\Omega,\closure_w\Omega}(f,g)=\langle f,\Delta_\Omega \,g\rangle_{L^2(\Omega)}+
\langle f,N_{\delta_w\Omega}\,g\rangle_{L^2(\delta_w\Omega)}=0+0=0
\]
for any $f\in L^2(\closure_w\Omega,\mu_\Lambda)$. Hence, $g\in\ker\Delta_w=\ker\mathcal{E}_\Omega$, because of Lemma \ref{DirichletLaplace}. In particular, if $(\Omega,w_\Omega)$ is a connected graphon, then $g$ is constant on $\Omega$. This proves the assertions.
\end{proof}

\begin{Proposition}
 Let $\phi\in L^2(\delta_w\Omega,\mu_\Lambda)$. If the Neumann boundary value problem for $(\phi,\Omega)$ has a
solution (weak or not), then
\[
\int_{\delta_w\Omega}\phi(x)\,d\mu_\Lambda(x) = 0
\]
holds true.
\end{Proposition}

\begin{proof}
Let $u\in \dom(\mathcal{E}_{\closure_w\Omega})$ be a solution (weak or not) of the Neumann boundary value problem for $(\phi,\Omega)$.
Using the $p$-adic Gauss-Green Identity (Lemma \ref{GaussGreen_p}), it follows that
\begin{align*}
\mathcal{E}_{\Omega,\closure_w\Omega}(1_{\closure_w\Omega},u)
&=\langle 1_{\closure_w\Omega}, \Delta_\Omega u\rangle_{L^2(\Omega,\mu_\Lambda)}
+\langle 1_{\closure_w\Omega},N_{\delta_w\Omega} u\rangle_{L^2(\delta_w\Omega,\mu_\Lambda)}
\\
&=\langle 1_{\closure_w\Omega},N_{\delta_w\Omega} u\rangle_{L^2(\delta_w\Omega,\mu_\Lambda)}
\\
&=\langle 1_{\closure_w\Omega},\phi\rangle_{L^2(\delta_w\Omega,\mu_\Lambda)}
\\
&=\int_{\delta_w\Omega}\phi(x)\,d\mu_\Lambda(x)\,.
\end{align*}
Since, by self-adjointness,
\[
\langle 1_{\closure_w\Omega},\Delta_\Omega u\rangle_{L^2(\Omega,\mu_\Lambda)}
=\langle \Delta_\Omega 1_{\closure_w\Omega},u\rangle_{L^2(\Omega,\mu_\Lambda)}
=\langle 0,u\rangle_{L^2(\Omega,\mu_\Lambda)}=0\,,
\]
the assertion now follows.
\end{proof}

Define the following space:
\[
\mathcal{N} = 
\mathset{u\in\dom(\mathcal{E}_{\closure_w\Omega})\cap L^\infty(\closure_w\Omega) \colon\Delta_w u\in L^\infty(X)}\,.
\]
\begin{Lemma}
The space $\mathcal{N}$ is dense in $\dom(\mathcal{E}_{\closure_w\Omega})$.
\end{Lemma}

\begin{proof} Adapting the proof of \cite[Lemma 3.7]{HS2022} to this more general setting is
as follows: Let $v\in\dom(\mathcal{E}_{\closure_w\Omega})$. Then
\[
f = (\Delta_\Omega + 1)v \in L^2(\closure_w\Omega,\mu_\Lambda)\,.
\]
Now, choose a sequence $(f_n)$ with $f_n\in L^\infty(\closure_w\Omega)$ such that
\[
\norm{f_n - f}_{L^2(\closure_w\Omega,\mu_\Lambda)}\to 0\,.
\]
Then we have $(\Delta_\Omega+1)^{-1}f_n\in\mathcal{N}$, and by continuity of the resolvent
operator that
\[
\norm{(\Delta_\Omega+1)^{-1}(f_n-f)}_{L^2(\closure_w\Omega,\mu_\Lambda)}\to0\,.
\]
It follows that
\begin{align*}
\norm{(\Delta_\Omega+1)^{-1}(f_n-f)}_{\mathcal{E}_{\closure_w\Omega}}^2
&=\mathcal{E}_{\closure_w\Omega}\left((\Delta_\Omega+1)^{-1}(f_n-f)\right)
\\
&+\norm{(\Delta_\Omega+1)^{-1}(f_n-f)}_{L^2(\closure_w\Omega,\mu_\Lambda)}^2
\end{align*}
tends to zero for $n\to\infty$. Therefore, $\mathcal{N}$ is dense in $\dom(\mathcal{E}_{\closure_w\Omega})$, as asserted.
\end{proof}

\begin{cor}\label{Existence_NBVP}
Let $\phi\in L^2(\delta_w\Omega)$ such that
\[
\int_{\delta_w\Omega}\phi(x)\,d\mu_\Lambda(x)= 0\,.
\]
Then $u\in L^\infty(\closure_w\Omega)$ is a weak solution of the Neumann Boundary Value Problem for $(\phi,\Omega)$, i.e.\ a solution of (\ref{NBVP_w}),  if
\[
\langle u,\Delta_\Omega v\rangle_{L^2(\closure_w\Omega,\mu_\Lambda)}=
\int_{\delta_w\Omega}\phi(x)v(x)\,d\mu_\Lambda(x)
\]
for every $v \in\mathcal{N}$.
\end{cor}

\begin{proof} 
We just follow the proof of \cite[Lemma 3.8]{HS2022}: First, $u\in\dom(\mathcal{E}_{\closure_w\Omega})$.
Namely, for every $\alpha>0$, we have that
\[
\Delta_\Omega\alpha(\Delta_\Omega + \alpha)^{-1} u 
=\alpha u - \alpha^2(\Delta_\Omega + \alpha)^{-1} u\in L^\infty(\closure_w\Omega)\,,
\]
which shows that $\alpha(\Delta_\Omega+\alpha)^{-1}u \in\mathcal{N}$. Hence, from the hypothesis on $u$, we
have
\begin{align*}
\int_{\delta_W\Omega}\phi\alpha(\Delta_\Omega+\alpha)^{-1}u\,d\mu_\Lambda
&=\langle u,\Delta_\Omega\alpha(\Delta_\Omega+\alpha)^{-1}u\rangle_{L^2}
\\
&=\alpha\langle u,u-\alpha(\Delta_\Omega+\alpha)^{-1}u\rangle_{L^2}
\end{align*}
for $\alpha>0$. Since $\mathcal{E}_{\closure_w\Omega}$ being Markov is equivalent with $\alpha(\Delta_\Omega+\alpha)^{-1}$ being
a contraction, cf.\ \cite[Theorem 1.4.1]{Fukushima1980}, this means that
\[
\int_{\delta_w\Omega}\phi\alpha(\Delta_\Omega+\alpha)^{-1}u\,d\mu_\Lambda
\le\mu_\Lambda(\delta_w\Omega)\norm{\phi}_\infty\norm{u}_\infty\,,
\]
and  thus
\[
\lim\limits_{\alpha\to\infty}\alpha
\langle u,u-\alpha(\Delta_\Omega+\alpha)^{-1}u\rangle_{L^2}<\infty\,.
\]
From the equality
\[
\dom(\mathcal{E}) =
\mathset{u\in L^2(\closure_w\Omega\mid \lim\limits_\alpha\langle u, u\rangle_{\mathcal{E}} < \infty}
\]
in \cite[Lemma 1.3.4]{Fukushima1980}, it now follows that $u\in\dom(\mathcal{E}_{\closure_w\Omega})$. By the assumption,
one has
\begin{align}\label{Dirichlet_E_int}
\mathcal{E}_{\closure_w\Omega}(u,v)&=
\langle u,\Delta_\Omega v\rangle_{L^2(\closure_w\Omega,\mu_\Lambda)}=\int_{\delta_w\Omega}\phi(x)v(x)\,d\mu_\Lambda(x)
\end{align}
for $v\in\mathcal{N}$. For general $v\in\dom(\mathcal{E}_{\closure_w\Omega})$ let $(v_n)_n\in\mathcal{N}$ such that
\[
\lim\limits_n\norm{v_n -v}_{\mathcal{E}_{\closure_w\Omega}} = 0\,.
\]
Then
\[
\lim\limits_{n}\mathcal{E}_{\closure_w\Omega}(u,v_n) 
=\mathcal{E}_{\closure_w\Omega}(u, v)\,,
\]
and since $\phi\in L^\infty(\closure_w\Omega,\mu_\Lambda)$, also
\[
\lim\limits_n\int_{\delta_w\Omega} \absolute{\phi v_n-\phi v}\,d\mu_\Lambda=0\,,
\]
which shows that (\ref{Dirichlet_E_int}) holds for all $v\in\dom(\mathcal{E}_{\closure_w\Omega})$. This now means that if $\supp(v)\subseteq\Omega$, then
\[
\langle \Delta_\Omega u,v\rangle_{L^2}=\langle u,\Delta_\Omega v\rangle_{L^2} = 0\,,
\]
and if $\supp(v)\subseteq\delta_w\Omega$, then
\[
\langle N_{\delta_w\Omega} u,v\rangle_{L^2}=\langle u,\Delta_\Omega v\rangle_{L^2} = \langle\phi,v\rangle_{L^2},
\]
i.e.\ $u$ is a weak solution of the Neumann Boundary Value Problem.
\end{proof}

\begin{remark} Corollary \ref{Existence_NBVP} thus gives a suffcient criterion for the existence
of a weak solution of the Neumann Boundary Value Problem (\ref{NBVP}) which is nothing but a Laplace equation under a von
Neumann boundary condition.
\end{remark}

\section{Green function}

\begin{definition}\label{ultracontractive}
Let $E$ be a locally compact separable metric space, and $m$ a
non-negative Radon measure on $E$ with full support. A semigroup $e^{-tL}$ is \emph{ultracontractive}, if there exists a decreasing function 
$\gamma\colon \mathds{R}_{>0}\to\mathds{R}_{>0}$ such that for
every $t > 0$ and $u\in L^2(E,m)$ we have
\[
\norm{e^{-tL} u}_{L^\infty(E,m)}
\le\gamma(t)\norm{u}_{L^2(E,m)}
\]
where the function $\gamma$ is called the \emph{rate function} of the semigroup.
\end{definition}

The overview in \cite[Appendix B]{HS2022} states the relevant relationships between Dirichlet forms, semigroups and Markov processes. A brief summary
is given here:
\newline

A Markovian Dirichlet form has an associated strongly continuous semi-
group with the Markovian property, cf. \cite[Theorem 1.4.1]{Fukushima1980}. In our case here,
the Dirichlet form $\mathcal{E}_{\closure_w\Omega}$ is closed and regular, i.e.\ $C(X)\cap\dom(\mathcal{E}_{\closure_w\Omega})$ is
uniformly dense in $C(X)$, as well as
$\norm{\cdot}_{\mathcal{E}_{\closure_w}}$-dense in 
$\dom(\mathcal{E}_{\closure_w\Omega} )$. 
Consequently, there exists a Hunt process $Y = (Y_t)_{t\ge0}$ on $\closure_w\Omega$, unique in a suitable
sense, such that for all Borel sets 
$A\subseteq\closure_w\Omega$ and for 
$\mu_\Lambda$-a.e.\ $x\in A$ we have
\[
e^{-t\Delta_\Omega}1_A(x)=\mathds{P}_x(Y_t\in A)\,.
\]
This is a Markov process $Y = (Y_t)_{t\ge0}$ associated with Dirichlet form $\mathcal{E}_{\closure_w\Omega}$,
and it holds true that
\[
e^{-t\Delta_\Omega}f(x)=\mathds{E}_xf(Y)=\int_{\closure_w\Omega}f(y)p_t(x,d\mu_\Lambda(y))
\]
for $\mu_\Lambda$-a.e.\ $x\in\closure_w\Omega$, $f\in L^2(\closure_w\Omega,\mu_\Lambda)$, and a heat kernel distribution
$p_t(x,\dot)$ valued in $[0,\infty)$ on the Borel sets of $\closure_w\Omega$. Under the condition of 
ultracontractivity, cf.\ Definition \ref{ultracontractive}, 
it turns out that $p_t(x,\cdot)$ also has a density function
\[
p_t(\cdot,\cdot)\colon\closure_w\Omega\times\closure_w\Omega\to[0,\infty)
\]
(named heat kernel) with the important properties
\begin{align}\label{hk_symmetric}
p_t(x,y)&=p_t(y,x)
\\
p_{t+s}(x,y)&=\int_{\closure_w\Omega}p_t(x,z)p_s(z,y)\,d\mu_\Lambda(z)
\end{align}
for $s,t > 0$ and $\mu_\Lambda$-a.e.\ $x,y \in\closure_w\Omega$.

\begin{definition}
The process $L = (L_t)_{t\ge0}$ with
\[
L_t =\int_0^1 1_{\mathset{Y_s\in\delta_w\Omega}}\, ds\,,
\]
where $t\ge 0$, is called the \emph{local time} of the process $Y$ on the boundary $\delta_w\Omega$.
\end{definition}

Notice that the local time $L$ satisfies
\begin{align}\label{localTimeDer}
(L_s)'=\frac{d}{dt}|_{t=s}L_t=1_{\mathset{Y_s\in\delta_w\Omega}}
\end{align}
almost everywhere on the half-line $s\ge 0$.

\begin{Lemma}\label{solutionFromheatKernel} 
Assume that the heat kernel $p_t(x,y)$ associated with the semigroup $e^{-t\Delta_\Omega}$ exists. Let $\phi\in L^\infty(\delta_w\Omega)$. For every $t\ge 0$, the function
\[
u_t(x) = \mathds{E}_x\int_0^t
\phi(Y_s)\, dL_s
\]
is bounded, and
\[
u_t(x) = \int_{\delta_w\Omega}\int_0^t
\phi(y)p_s(x,)\,ds\, d\mu_\Lambda(y)
\]
for $x\in\closure_w\Omega$ holds true.
\end{Lemma}

\begin{proof}
Inspired by the proof of \cite[Lemma 3.10]{HS2022}, we see that
\begin{align*}
u_t(x)&=\mathds{E}_x\int_0^t\phi(Y_s)(L_s)'\,ds=\mathds{E}_x\int_0^t\phi(Y_s)1_{Y_s\in\delta_w\Omega}\,ds
\\
&=\int_{\delta_w\Omega}\mathds{E}_x\phi(y)\delta_y(Y_s)\,ds\,d\mu_\Lambda(y)
\\
&=\int_{\delta_w\Omega}\phi(y)\mathds{E}_x\int_0^t\delta_y(Y_s)\,ds\,d\mu_\Lambda(y)
\\
&\stackrel{(*)}{=}\int_{\delta_w\Omega}\phi(y)\int_0^t\mathds{E}_x\delta_y(Y_s)\,ds\,d\mu_\Lambda(y)
\\
&=\int_{\delta_w\Omega}\phi(y)\int_0^t\delta_y(z)p_s(x,d\mu_\Lambda(z))\,ds\,d\mu_\Lambda(y)
\\
&=\int_{\delta_w\Omega}\phi(y)\int_0^t p_s(x,y)\,ds\,d\mu_\Lambda(y)
\\
&=\int_{\delta_w\Omega}\int_0^t\phi(y)p_s(x,y)\,ds\,d\mu_\Lambda(y)
\end{align*}
for $x \in\closure_w\Omega$, which admits the bound
\[
\absolute{u_t}\le t\norm{\phi}_\infty
\]
for $t\ge0$. The first equality holds true because $L_s$ is almost everywhere
differentiable on the half-line $s\ge 0$, cf.\ (\ref{localTimeDer}). Equality $(*)$ follows from
Fubini-Tonelli.
\end{proof}

\begin{remark}\label{ultracontractivityProperties}
From \cite[Appendix B]{HS2022}, we take the following information:
Let $E$ be a locally compact separable metric space, and $m$ a non-negative
Radon measure on $E$ with full support.
\begin{enumerate}
\item If $(e^{-tL})_{t>0}$ is ultracontractive with rate function $\gamma$ and $\mathset{p_t(\cdot,\cdot)}_{t>0}$ is a heat
kernel for $(e^{-tL})_{t>0}$, then we have
\[
p_{2t}(x,y)\le\gamma(t)^2
\]
for every $t > 0$ and $m$-a.e.\ $x,y\in E$, cf.\ \cite[Lemma 2.1.2.]{Davies1989}.

\item If $m(E) < \infty$ and $(e^{-tL})_{t>0}$ is ultracontractive, then the generator $L$ has
purely discrete spectrum, see for instance \cite[Theorem 2.1.4.]{Davies1989}. In this case the
heat kernel shows a typical mixing property: For every $t_0 > 0$ there are constants
$c_1,c_2 > 0$ such that the inequality
\begin{align}\label{ultracontractive_ineq}
\absolute{p_t(x,y) -\frac{1}{m(E)}
}
\le c_1 e^{-c_2 t}
\end{align}
holds for every $t > t_0$ and $m$-a.e.\ $x,y\in E$, cf.\ \cite[Eq. (18)]{HS2022}.
\end{enumerate}
\end{remark}

\begin{example}\label{geodistUltracont}
The kernel function
\[
w(x,y) = d_\Lambda(x,y)^{-\alpha}
\]
with $\alpha > 0$ yields an example of an ultracontractive semigroup on a compact $p$-adic analytic $n$-manifold $X$ with integral structure $\Lambda$ determined by the geodetic
distance $d_\Lambda$. In this case, the rate function is the constant function $\gamma(t) = 1$.
\end{example}

If a semigroup on $\closure_w\Omega$ has a heat kernel $p_t(x,y)$, then its associated
\emph{Green function} 
can be defined as
\begin{align}\label{GreenFunctionOmega}
G_\Omega(x,y)=\int_0^\infty\left(
p_t(x,y)-\sum\limits_{\psi\colon\lambda_\psi=0}
\psi(x)\overline{\psi(y)}
\right)\,,
\end{align}
where $\psi$ in the sum runs through an orthonormal basis of the null space of
$L$, and with $x,y \in\closure_w\Omega$. In the case of ultracontractivity and connectedness
of $\closure_w\Omega$, the existence of the Green function follows from (\ref{ultracontractive_ineq}).
\newline

The  heat
equation satisfied by the heat kernel $p_t(x,y)$ is the following one:
\begin{align}\label{heatEquation}
\frac{d}{ds}p_s(\cdot, y)|_{s=t} = -\Delta_\Omega p_t (\cdot, y)
\end{align}
for $y \in\closure_w\Omega$ $t \ge 0$.

\begin{thm}\label{SolutionOfNBVP_p}
 Let $X$ be a compact $p$-adic analytic $n$-manifold, and $\Delta_w$ an
integral operator with kernel function $w(x,y)$, which is the generator of an ultracontractive semigroup on a connected subdomain $\closure_w\Omega X$ with $\Omega\subseteq X$
open. 
Let $\phi\in L^\infty(\delta_w\Omega)$ such that
\[
\int_{\delta_w\Omega}
\phi\, d\mu_\Lambda = 0\,.
\]
Then the following statements hold true:
\begin{enumerate}
\item There is a unique function $u \in L^\infty(\closure_w\Omega)$ that weakly solves the Neumann Boundary Value Problem for $(\phi,\Omega)$, and satisfies
\[
\int_{\closure_w\Omega}u(x)\,d\mu_\Lambda(x) = 0 .
\]
\item This function satisfies
\begin{align*}
u(x)&=\lim\limits_{t\to\infty}\mathds{E}_x\int_0^t(Y_s)\,dL_s
\\
&=\int_0^\infty\int_{\delta_w\Omega}\phi(y)p_s(x,y)\,d\mu_\Lambda(y)\,ds
\\
&=\int_{\delta_w\Omega}\phi(y)G_\Omega(x,y)\,d\mu_\Lambda(y)
\end{align*}
for all $x\in\closure_w\Omega$.
\end{enumerate}
\end{thm}

\begin{proof}
We just follow the lines of the proof of \cite[Theorem 3.12]{HS2022}.

\smallskip
We first show that $u(x) := \lim\limits_{t\to\infty} u_t (x)$ exists. By Lemma \ref{solutionFromheatKernel}, and since
$\phi$ is centered,
\begin{align*}
u_t(x)&=\int_{\delta_w\Omega}\int_0^t\phi(y)p_s(x,y)\,ds\,d\mu_\Lambda(y)
\\
&=\int_{\delta_w\Omega}\int_0^t\phi(y)\left(p_s(x,y)-\frac{1}{\mu_\Lambda(\closure_w\Omega)}\right)\,ds\,d\mu_\Lambda(y)\,,
\end{align*}
and given $0 < s < t$, we can see that
\begin{align*}
\absolute{u_t(x)-u_s(x)}&\le\int_{\delta_w\Omega}\int_s^t\absolute{\phi(y)}\absolute{p_r(x,y)-\frac{1}{\mu_\Lambda(\closure_w\Omega)}}\,dr\,d\mu_\Lambda(y)
\\
&\le c_1\int_{\delta_w\Omega}\absolute{\phi(y)}\,d\mu_\Lambda(y)\int_s^te^{-c_2r}\,dr\,,
\end{align*}
using (\ref{ultracontractive_ineq}).
This expression becomes arbitrarily small for $s, t$ are large. Hence,
the limit function $u$ exists, is bounded, and satisfies
\[
u(t)=\int_0^\infty\int_{\delta_w\Omega}\phi(y)p_s(x,y)\,d\mu_\Lambda(y)\,ds
\]
for $x\in\closure_w\Omega$. By the expression (\ref{GreenFunctionOmega}) for the Green function and since $\phi$ is
centred, it follows that
\[
u(x)=\int_{\delta_w\Omega}\phi(y)G_\Omega(x,y)\,d\mu_\Lambda(y)
\]
or every $x\in\closure_w\Omega$. Integrating this yields
\[
\int_{\closure_w\Omega}u(x)\,d\mu_\Lambda(x)=\int_0^\infty\int_{\delta_w\Omega}\phi(y)\int_{\closure_w\Omega}p_s(x,y)\,d\mu_\Lambda(x)\,d\mu_\Lambda(y)\,ds\,,
\]
and so
\[
\int_{\closure_w\Omega}u(x)\,d\mu_\Lambda(x)=0\,,
\]
since $\phi$ is centred, and
\[
\int_{\closure_w\Omega}p_s(x,y)\,d\mu_\Lambda(y)=1\,,
\]
which can be used since $p_s(x, y)$ is symmetric in $x,y\in\closure\Omega$, cf.\ (\ref{hk_symmetric}). We show that
$u$ weakly solves the Neumann problem by an application of Corollary \ref{Existence_NBVP}.
Let $v\in\mathcal{N}$ be arbitrary. We have to show
\[
\langle u,\Delta_\Omega v\rangle_{L^2(\closure_w\Omega,\mu_\Lambda)}=\int_{\delta_w\Omega}\phi(x)v(x)\,d\mu_\Lambda(x)\,.
\]
The uniform convergence of $u_t$ to $u$ yields
\[
\langle u,\Delta_\Omega v\rangle_{L^2(\closure_w\Omega)}=\lim\limits_{t\to\infty}\langle u_t,\Delta_\Omega v\rangle_{L^2(\closure_w\Omega,\mu_\Lambda)}\,.
\]
On the other hand,
\begin{align*}
\langle u_t,\Delta_\Omega v\rangle_{L^2(\closure_w\Omega,\mu_\Lambda)}&=
\int_{\delta_w\Omega}\int_0^t\phi(y)\int_{\closure_w\Omega}p_s(x,y)\Delta_\Omega v(x)\,d\mu_\Lambda(x)\,d\mu_\Lambda(y)\,ds
\\
&=\int_{\delta_w\Omega}\int_0^t\phi(y)\langle p_s(\cdot,y),\Delta_\Omega v\rangle_{L^2(\closure_w\Omega,\mu_\Lambda)}\,d\mu_\Lambda(y)\,ds\,.
\end{align*}
Using of the self-adjointness of $\Delta_\Omega$ and the heat equation (\ref{heatEquation}), obtain
\[
\langle p_s(\cdot,y),\Delta_\Omega v\rangle_{L^2(\closure_w\Omega,\mu_\Lambda)}=-\left\langle
\frac{d}{dt}p_t(\cdot,y)|_{t=s}
\right\rangle_{L^2(\closure_w\Omega,\mu_\Lambda)}\,,
\]
and  therefore
\begin{align*}
\langle p_s&(\cdot,y),\Delta_\Omega v\rangle_{L^2(\closure_w\Omega,\mu_\Lambda)}
\\
&=-\int_{\delta_w\Omega}\int_0^t\int_{\closure_w\Omega}\phi(y)\left(\frac{d}{dt}p_t(\cdot,y)|_{t=s}\right)(x)v(x)\,d\mu_\Lambda(x)\,d\mu_\Lambda(y)\,ds
\\
&=\int_{\delta_w\Omega}\int_{\closure_w\Omega}\phi(y)(p_0(x,y)-p_t(x,y))v(x)\,d\mu_\Lambda(x)\,d\mu_\Lambda(y)\,.
\end{align*}
By Lebesgue’s theorem and (\ref{ultracontractive_ineq}), it now follows that
\begin{align*}
\lim\limits_{t\to\infty}&\int_{\delta_w\Omega}\int_{\closure_w\Omega}\phi(y)p_t(x,y)v(x)\,d\mu_\Lambda(x)\,d\mu_\Lambda(y)
\\
&=\int_{\delta_w\Omega}\phi(y)\,d\mu_\Lambda(y)\frac{1}{\mu_\Lambda(\closure_w\Omega)}\int_{\closure_w\Omega}v(x)\,d\mu_\Lambda(x)\,,
\end{align*}
so that
\begin{align*}
\langle u,\Delta_\Omega v \rangle_{L^2(\closure_w\Omega,\mu_\Lambda)}&=\int_{\delta_w\Omega}\int_{\closure_w\Omega}\phi(y)p_0(x,y)v(x)\,d\mu_\Lambda(x)\,d\mu_\Lambda(x)
\\
&=\int_{\delta_w\Omega}\phi(y)v(y)\,d\mu_\Lambda(y)\,.
\end{align*}
The result now follows from Corollary \ref{Existence_NBVP}.
\end{proof}

\section{Locating zeros and poles of analytic functions}

Let 
$f\colon X\to K$ be a holomorphic function. The goal here is to identify the zero locus $V(f)\subset X$ via an NVP, as well as this is possible.
First, adapt the kernel function by incorporating $f$:
\[
\tilde{w}(x,y)=w(x,y)\absolute{f(y)}\,.
\]
This amounts to replacing $d\mu_\Lambda$ with the following measure:
\[
d\nu_f(y)=\absolute{f(y)}\,d\mu_\Lambda(y)\,,
\]
and thus obtain a new operator
\begin{align*}
\Delta_{w,f,\Omega} u(x)&=\int_{\closure_w\Omega}\tilde{w}(x,y)(u(x)-u(y))\,d\mu_\Lambda(y)
\\
&=\int_{\closure_w\Omega}
w(x,y)(u(x)-u(y))\absolute{f(y)}\,d\mu_\Lambda(y)
\\
&=\int_{\closure_w\Omega}w(x,y)(u(x)-u(y))\,d\nu_f(y)\,.
\end{align*}
The
specific case:
\[
w(x,y)=d_\Lambda(x,y)^{-n\alpha}\,,
\]
whose corresponding operator is denoted as $\Delta^\alpha_{f,\Omega}$,
is of concern in what follows.

\begin{Proposition}\label{smallWaveletEigenvalue}
The wavelets $\psi$ on $\closure_w\Omega$ supported in some sufficiently small
$B(a)\subset X\setminus V(f)$ are eigenfunctions of $\Delta_{f,\Omega}^\alpha$ with eigenvalue
\[
\lambda_\psi=\int_{\closure_w\Omega\setminus B(a)}d_\Lambda(x,y)^{-\alpha}\absolute{f(y)}\,d\mu_\Lambda(y)
+\nu_f(B(a))^{-n\alpha}\left(1-q^{-n}(1+(-1)^n\right)
\]
and is independent of $x\in B(a)$.
\end{Proposition}

\begin{proof}
Since $B(a)\cap V(f)=\emptyset$ and $B(a)$ is small, the important vanishing integral property
\[
\int_{\closure_w\Omega}\psi(x)\,d\nu(x)=0
\]
holds true, which can be seen as follows: locally, on a chart $U\subset \closure_w\Omega$ containing $B(a)$, we have
\begin{align}\label{polynomialMeasure}
d\nu(x)=\absolute{g_U(x)}\absolute{f(x)}\absolute{dx}=\absolute{g_U(x)f(x)}\absolute{dx}
=\absolute{P_U(x)}\absolute{dx}\,,
\end{align}
where $P_U(x)$ is a polynomial with coefficients in $K$ in $n$ variables, according to the Weierstrass preparation theorem. 
Notice that, by the
strict triangle inequality, this implies that
\[
z\mapsto \absolute{P(z)}
\]
defines a locally constant function on $U$ outside the zeros of $f$. Hence, if the the $p$-adic ball
$B(a)\subset U$ is sufficiently small, then $\absolute{P(z)}$ is constant on $B(a)$. In this
case,
\begin{align}\nonumber
\int_{\closure_w\Omega}\psi(x)\,d\nu_f(y)&=
\int_{\closure_w\Omega}\psi(x)\absolute{f(x)}{d\mu_\Lambda(x)}
\\\label{waveletZeroMean}
&=\int_U\psi(x)\absolute{f(x)g_U(x)}\absolute{dx}
=\tilde{C}\int_U\psi(x)\absolute{dx}=0
\end{align}
for some $\tilde{C}>0$, where the latter equality holds true, as
the vanishing of the Kozyrev wavelet on any measurable
set containing its support is well known, for example, cf.\ \cite[Theorem 2.3.9]{XKZ2018}. Notice that, in the context of the sloppy notation, the set $U$ is identified with its image in $K^n$ under the chart map, and $\psi$ is there identified with a Kozyrev wavelet.

\smallskip
For the eigenvalue formula, first observe that the independence of the
choice of $x\in B(a)$ is immediate. Next, observe that this formula is an extension to higher dimension of the one in \cite[Theorem 3]{Kozyrev2004}, and the proof of that uses only ultrametricity within the support of the wavelet. So, in this case, if $x \in \closure_w\Omega\setminus B(a)$, then
\begin{align}\nonumber
\Delta^\alpha_{f,\Omega}\psi(x)&=\int_{\closure_w\Omega}d_\Lambda(x,y)^{-n\alpha}(\psi(x)-\psi(y))\,d\nu_f(y)
\\\label{useVanishInt}
&=\psi(x)\int_{\closure_w\Omega} d_\Lambda(x,y)^{-n\alpha}\,d\nu_f(y)
-d_\Lambda(x,B(a))^{-n\alpha}\int_{\closure_w\Omega}\psi(x)\,d\nu_f(y)
\\\nonumber
&=0\,,
\end{align}
because  the latter integral vanishes by (\ref{waveletZeroMean}), and as $\psi(x)=0$. If $x\in B(a)$, then
\begin{align*}
\Delta^\alpha_{f,\Omega}\psi(x)&=\psi(x)\int_{X\setminus B(a)}d\Lambda(x,y)^{-n\alpha}\,d\nu_f(y)
+\int_{B(a)}d_\Lambda(x,y)^{-n\alpha}(\psi(x)-\psi(y))\,d\nu_f(y)\,.
\end{align*}
The latter integral equals
\begin{align*}
I&=\int_{\partial B(a)}d_\Lambda(x,y)^{-n\alpha}(\psi(x)-\psi(y))\absolute{f(y)}\,d\mu_\Lambda(y)\,.
\end{align*}
Again, since $B(a)$ is small, on a local chart $(U,\phi)$ containing $B(a)$, equality (\ref{polynomialMeasure}) shows that
\[
d\nu_f(y)|_{B(a)}=C\absolute{dy}
\]
for some $C>0$. Hence,
\begin{align*}
I&=C\int_{\partial \phi(B(a))}\norm{\phi(x)-z}_{K^n}(\psi_{\phi(B(a)),j}(x)-\psi_{\phi(B(a)),j}(y))\absolute{dy}\,,
\end{align*}
where $\psi_{\phi(B(a)),j}$ is a Kozyrev wavelet supported in $\phi(B(a))$, and depending on $j\in\mathds{F}_q^\times$. Assume that the radius of $B(a)$ is $q^{-k}$. Then points $z\in\partial\phi(B(x))$ have the form
\[
z=\phi(x)+u+\epsilon
\]
with $u=(u_1,\dots,u_n)\in\pi^kO_K^n$ such that for $i\in\mathset{1,\dots,n}$:
\[
u_i=\pi^k\tau_i(\ell_i)\,,\quad\ell_i\in\mathds{F}_q
\]
and $\ell_i\in\mathds{F}_q^\times$
 for at least one $i\in\mathset{1,\dots,n}$, and with $\epsilon\in\pi^{k+1}O_K^n$. 
 The maps
 \[
\tau_i\colon \mathds{F}_q\to O_K
 \]
 are lifts of the canonical map $O_K\to\mathds{F}_q$ in the coordinates $i=1,\dots,n$, forming the map $\tau\colon \mathds{F}_q^n\to O_K^n$. Hence,
 \begin{align*}
I&=C\sum\limits_{i=1}^n\sum\limits_{\ell=(\ell_1,\dots,\ell_n)\in\atop\mathds{F}_q\times\dots\times\mathds{F}_q^\times\times\dots\times\mathds{F}_q}
\int_{\pi^{k+1}O_K^n}\norm{\tau(\ell)+\epsilon}^{-n\alpha}
\left[\psi_{\phi(B(a)),j}(\phi(x))\right.
\\
&\quad\left.-\psi_{\phi(B(a)),j}(\phi(x)+\tau(\ell)+\epsilon)\right]\,\absolute{d\epsilon}
\\
&=Cq^{\alpha nk-n(k+1)}\sum\limits_{\ell\in\mathds{F}_q^n\setminus\mathset{\bar{0}}}\left[\psi_{\phi(B(a)),j}(\phi(x))-\psi_{\phi(B(a)),j}(\phi(x)+\tau(\ell))\right]
\\
&\stackrel{(*)}{=}
Cq^{\alpha nk-n(k+1)}\sum\limits_{\ell\in\mathds{F}_q^n\setminus\mathset{\bar{0}}}
\left[1-\chi(\pi^{-1}\tau(j)\tau(\ell))\psi_{\phi(B(a)),j}(\phi(x))\right]
\\
&\stackrel{(**)}{=}
Cq^{\alpha nk-n(k+1)}\left(q^n-(1+(-1)^n)\right)\psi_{\phi(B(a)),j}(\phi(x))
\\
&=C\mu(\phi(B(a)))^{n-\alpha}\left(1-q^{-n}(1+(-1)^n)\right)\psi_{\phi(B(a)),j}(\phi(x))
\\
&=\nu_f(B(a))^{n-\alpha}\left(1-q^{-n}(1+(-1)^n)\right)\psi(x)\,,
 \end{align*}
 where $(*)$ follows by inspecting the Kozyrev wavelets, and $(**)$ from
 \[
\sum\limits_{\ell\in\mathds{F}_q^n\setminus\mathset{\bar{0}}}\left(1-\chi(\pi^{-1}\tau(j)\tau(\ell))\right)
=q^n-(1+(-1)^n)\,,
 \]
 as can be seen by induction. This now implies the asserted eigenvalue.
\end{proof}

In the proof of Proposition \ref{smallWaveletEigenvalue}, it was deemed helpful to use in the end a less sloppy notation for integrating over differential forms than is often the case.

\begin{example}
Let $n=1$, $K=\mathds{Q}_p$ with $p\ge3$. Let $X=\mathds{Z}_p\cup (p^{-1}+\mathds{Z}_p)$, on which the differential $1$-form
\[
\omega(x)=\begin{cases}
(x-p^2)(x+p^2)\,dx,&x\notin B_3(p^2)\cup B_3(-p^2)
\\
dx,&\text{otherwise}
\end{cases}
\]
is defined,
and let 
\[
\Delta_\omega^\alpha u(x)
=\int_{\mathds{Z}_p}\absolute{x-y}^{-\alpha}(u(x)-u(y))\,\absolute{\omega(x)}
\] 
be, up to a constant factor, the Vladimirov operator on $\mathds{Z}_p$.
The Kozyrev wavelet
\[
\psi(x)=e^{\{i p^{-1}x\}_p}1_{\mathds{Z}_p}(x)\,,
\]
where 
\[
\{ \alpha_{-m}p^{-m}+\alpha_{-m+1}p^{-m+1}+\dots\}_p=\sum\limits_{j=-m}^{-1}\alpha_jp^j\in\mathds{Q}\,,\quad\alpha_j\in\mathset{0,\dots,p-1}\,,
\]
is the principal part of a $p$-adic number,
is not an eigenfunction of $\Delta_\omega^\alpha$.
\end{example}

\begin{proof}
We check  that
\begin{align*}
\int_{\mathds{Z}_p}\psi(x)\absolute{\omega(x)}
&=\int_{B_3(p^2)}1\absolute{dx}
+\int_{B_3(-p^2)}1\absolute{dx}
\\
&+\int_{\absolute{x}=1}\psi(x)\absolute{dx}
+\int_{\absolute{x}=p^{-1}}1\absolute{dx}
+\underbrace{\int_{H}1\absolute{x-p^2}_p\absolute{x+p^2}_p\absolute{dx}}_{=I}
\\
&=p^{-2}+p^{-2}-p^{-1}+p^{-1}(1-p^{-1})+I
\\
&=I+p^{-2}\,,
\end{align*}
where
\[
H=\mathset{\absolute{x}=p^{-2}}\setminus (B_3(p^2)\cup B_3(-p^2))\,.
\]
Since $I>0$, it follows that
\[
\int_{\mathds{Z}_p}\psi(x)\absolute{\omega(x)}\neq0\,.
\]
Hence, $\psi$ is not an eigenfunction of $\Delta_\omega^\alpha$, because
\[
\Delta_\omega^\alpha\psi(x)=-I\neq0
\]
for $x\in
X\setminus\mathds{Z}_p$
by the calculation in (\ref{useVanishInt}).
\end{proof}

Since, according to Example \ref{geodistUltracont}, the semigroup $e^{-t\Delta_\Omega}$ is ultracontractive, and $\nu_f(X)<\infty$, Remark \ref{ultracontractivityProperties}.2.\ says that the spectrum of $\Delta_{f,\Omega}^\alpha$ is a point spectrum, and thus its heat kernel can be written as
\[
p_t^\Omega(x,y)=\sum\limits_\psi e^{-\lambda_\psi t}\psi(x)\overline{\psi(y)}\,,
\]
where $\psi$ runs through an orthonormal basis of $L^2(\closure_w\Omega,\nu_f)$ consisting of eigenfunctions of $\Delta_{f,\Omega}^\alpha$.
The corresponding Green function is then of this form:
\[
G_{\Omega,f}(x,y)=\sum\limits_{\psi\colon\lambda_\psi>0}\frac{1}{\lambda_\psi}\psi(x)\overline{\psi(y)}
\]
for $x,y\in\closure_w\Omega$. As an application, solve the following Neumann Boundary Value Problem for $(\phi_k,\Omega_{a,k})$:
\begin{align}\label{NBVP_f}
\Delta_{f,\Omega_{a,k}}^\alpha u|_{\Omega_{a,k}}=0,\quad N_{\delta_w\Omega}u|_{\delta_w\Omega}u=\phi
\end{align}
with
\begin{align*}
\Omega_{a,k}&=X\setminus (V(f)\cup B_k(a)),&
B_k(a)\subset X\setminus V(f)
\\
\delta_w\Omega_{a,k}&=B_k(a)\cup V(f)
\\
\phi_k(x)&=\psi_{B_k(a),j}(x)&\text{(wavelet with parameter $j\in\left(\mathds{F}_p^\times\right)^n$).}
\end{align*}
Its solution is
\[
u(x)=\int_{\delta_w\Omega}\phi(y)G_{\Omega,f}(x,y)\,d\nu(y)=\frac{1}{\lambda_{\psi_{B_k(a)}}}\psi_{B_k(a)}(x)\in\dom(\mathcal{E}_{f,\closure_w\Omega})\,,
\]
cf.\ Theorem \ref{SolutionOfNBVP_p}.2., where the used Dirichlet form $\mathcal{E}_{f,\closure_w\Omega}$ uses the Radon measure $\nu_f$. Notice that the semigroup $e^{-t\Delta_{f,\Omega,}^\alpha}$ is ultracontractive for $\alpha>0$, just like in Example \ref{geodistUltracont}.
\newline

The decay of 
$\norm{u}_\infty\to 0$ for the solution $u\in\dom(\mathcal{E}_{f,\closure_w\Omega})$ of (\ref{NBVP_f}) tells us how close $a$ is to $V(f)$:

\begin{thm}[Closeness to zeros]
The norm $\norm{u}_\infty$ of the solution $u\in\dom(\mathcal{E}_{f,\closure_w\Omega})$ of the Neumann boundary value problem (\ref{NBVP_f})
informs for varying $a\in X\setminus V(f)$ about whether or not $a$ is approaching some $z_0\in V(f)$ of order $r>0$.
\end{thm}

\begin{proof}
Let $z_0\in V(f)$ be of order $r>0$, and assume that $a\in X$ is closer to $z_0$ than to any other point in $V(f)$. Then
\[
\nu_f(B(a))=C(a)\absolute{a-z_0}^r q^{-k}
=C(a)q^{-(k+rn_a)}
\]
for some $C(a)>0$ bounded from below, where
$\absolute{a-z_0}=q^{-n_a}$.
According to Proposition \ref{smallWaveletEigenvalue},
\[
\lambda_{\psi_{B_k(a)},j}\in O(\nu_f(B_k(a))^{-n\alpha})\,.
\]
Hence, $\norm{u}_\infty=\frac{1}{\lambda_\psi}$ is of the order
\[
Cq^{-(k+rn_a)n\alpha}=Cq^{-kn\alpha}(q^{-n\alpha})^{rn_a}
\]
for some $C>0$.
This means that for fixed radius of a small ball $B_k(a)$, the (existent or non-existent) rate of the exponential decay of the solution $\norm{u}_\infty$ informs about an approximate location of $z_0\in V(f)$, as well as its order $r>0$, as $a$ approaches $z_0$ (or not). 
\end{proof}

\section*{Acknowledgements}
Peer Kunstmann,
\'Angel M\'oran Ledezma and David Weisbart are warmly thanked for fruitful discussions.
This research is partially supported by
the Deutsche Forschungsgemeinschaft under project number 469999674.

\bibliographystyle{plain}
\bibliography{biblio}

\end{document}